\documentclass[a4paper,12pt,oneside]{article}
\usepackage{hyperref}
\usepackage[T1]{fontenc}
\usepackage{authblk} 
\usepackage{mathtools}
\usepackage{mathrsfs}
\usepackage{amssymb}
\usepackage{xifthen}
\usepackage{amsmath}
\usepackage{amsthm}
\usepackage{latexsym}
\usepackage[all]{xy}
\usepackage{amscd}

\usepackage{geometry}
\newgeometry{
  top=28mm,
  bottom=28mm,
  left=25mm,
  right=10mm
}

\hypersetup{
    colorlinks,
    linkcolor={red},
    citecolor={green},
    urlcolor={blue}
}

\newtheorem{theorem}{Theorem}
\newtheorem{lemma}[theorem]{Lemma}
\newtheorem{corollary}[theorem]{Corollary}

\theoremstyle{definition}

\newtheorem{conjecture}[theorem]{Conjecture}

\DeclareMathOperator{\Aut}{Aut}

\usepackage{verbatim}
\usepackage{xcolor}
\usepackage{authblk}
\hypersetup{
    colorlinks,
    linkcolor={orange},
    citecolor={magenta},
    urlcolor={blue}
}

\title{Solvable groups defined by redundant cyclic presentations}

\author{Ihechukwu Chinyere\thanks{Corresponding author: \texttt{i.chinyere@up.ac.za; ihechukwu@aims.ac.za}}}
\affil{Department of Mathematics and Applied Mathematics, University of Pretoria, Hatfield 0028, Pretoria, South Africa}

\begin{document}
\maketitle

\begin{abstract}
\noindent
We study the two-generator one-relator groups defined by $\langle x_0,x_1\mid U(x_0,x_1)U(x_1,x_0)^{-1}\rangle$, and determine exactly which of them fail to contain a free subgroup of rank $2$. We show that every such group is a solvable Baumslag--Solitar group $BS(1,m)$, and that every integer value $m$ can occur. This disproves a conjecture predicting that only $m\in \{-1,0,1\}$ can occur.

\vspace{1em}

\end{abstract}

\section{Introduction}
\label{sec:intro}

Let \(F_n=F(x_0,\ldots,x_{n-1})\) be the free group of rank \(n\geq 1\), and let \(\theta:F_n\to F_n\) be the shift automorphism defined by \(\theta(x_i)=x_{i+1}\), where subscripts are taken modulo \(n\). Given a non-empty cyclically reduced word \(w\in F_n\), the presentation
$$
P_n(w)=\left\langle x_0,\ldots,x_{n-1}\,\middle|\,w,\theta(w),\ldots,\theta^{n-1}(w)\right\rangle
$$
is called a \emph{cyclic presentation}. The group \(G_n(w)\) it defines is called a \emph{cyclically presented group}, and \(w\) is called its \emph{defining word}~\cite{Johnson1997}.

\medskip
Following \cite[page 160]{Bogley2014}, given a group presentation \(P=\langle X\mid R\rangle\), an element \(r\in R\) is said to be \emph{freely redundant} if it is freely trivial or if there exists another element \(s\in R\) such that \(r\) and \(s\) are elements of the free group with basis \(X\) and either \(r\) is freely conjugate to \(s\) or \(r\) is freely conjugate to \(s^{-1}\). A presentation is said to be \emph{redundant} if it contains a freely redundant relator and \emph{concise} otherwise. A cyclic presentation \(P_n(w)\) is \emph{orientable} if \(w\) is not a cyclic permutation of the inverse of any of its shifts \cite[page 155]{Bogley2014}, and \emph{non-orientable} otherwise. Redundant cyclic presentations were studied and classified in \cite{ChinyereWilliams2023} (see also \cite{ChinyereWilliams2022}). It follows from the classification that $P_2(w)$ is redundant, then $w=U(x_0,x_1)U(x_1,x_0)^\epsilon$, where $\epsilon=1$ in the orientable case and $\epsilon=-1$ in the non-orientable case. Hence, $w(x_1,x_0)$ can be removed from the presentation, so we get a $2$-generator-$1$-relator group.

\medskip
A classical dichotomy, conjectured by Magnus~\cite{Magnus1932} and proved by Moldavanski\u{\i}~\cite{Moldavanskii1969}, states that a one-relator group contains no non-abelian free subgroup if and only if it is cyclic or a Baumslag--Solitar group $
BS(1,k)=\langle a,t\mid tat^{-1}=a^k\rangle.$ The groups $BS(1,k)$ belong to the Baumslag--Solitar family introduced by Baumslag and Solitar~\cite{BaumslagSolitar1962} as examples of non-Hopfian one-relator groups. Chebotar~\cite{Chebotar1971} subsequently classified the subgroups of one-relator groups that contain no non-abelian free subgroup.

\begin{theorem}[Chebotar~\cite{Chebotar1971}]\label{thm:chebotar}
If the subgroup $C$ of the group $G$ with one defining relation
contains no free subgroup of rank $2$, then either it is Abelian or it
has one of the following presentations:
\[
C=\langle a,b\mid a^2=b^2\rangle,
\]
\[
C=\langle a,b\mid aba^{-1}=b^\kappa\rangle,
\]
where $\kappa$ is an integer.
\end{theorem}
In particular, one-relator groups satisfy the corresponding Tits alternative: every subgroup either contains a non-abelian free subgroup of rank $2$ or is solvable.

\medskip
Specialising this Tits alternative to redundant cyclic presentations,  it was shown in \cite{ChinyereWilliams2023} that if $P_n(w)$ is redundant with $n\geq 2$, then $G_n(w)$ either contains a non-abelian free group or is solvable. In the orientable case, they provide a complete classification of precisely those cases in which the solvable alternative occurs.

\begin{corollary}\cite[Corollary~4.3]{ChinyereWilliams2023}
Let $P_n(w)$ be an orientable redundant cyclic presentation. Then $G_n(w)$ contains a non-abelian free group unless $n=2$ and $w$ is a cyclic permutation of $x_0^\epsilon x_1^\epsilon$ for some $\epsilon=\pm1$, in which case $G_2(w)\cong\mathbb{Z}$, or $w$ is a cyclic permutation of $x_0^{2\epsilon}x_1^{2\epsilon}$ for some $\epsilon=\pm1$, in which case $G_2(w)\cong BS(1,-1)$.
\end{corollary}

They then turn to the complementary, non-orientable case, first disposing of the degenerate situation in which \(U\) is a proper power.

\begin{lemma}\cite[Lemma~4.4]{ChinyereWilliams2023}\label{lem:power}
 Let \(u(x_0,x_1)\) be a reduced word and let \[G=\langle x_0,x_1\mid u(x_0,x_1)u(x_1,x_0)^{-1}\rangle,\] where \(u(x_0,x_1)=v(x_0,x_1)^t\) for some \(t\geq2\). Then \(G\) contains a non-abelian free subgroup except if \(t=2\) and either \(v(x_0,x_1)=x_0^{\pm1}\) or \(v(x_0,x_1)=x_1^{\pm1}\), in which case \(G\) is isomorphic to \(BS(1,-1)\).

\end{lemma}

Hence, whenever $U$ is a proper power, the non-orientable case is already completely settled, and contributes exactly the value $m=-1$. On the strength of this, together with the orientable classification of Corollary~4.3, it was conjectured that the same property persists for every non-orientable presentation.

\begin{conjecture}\cite[Conjecture~4.5]{ChinyereWilliams2023}\label{conj:cw}
If $P_2(w)$ is a non-orientable cyclic presentation, then $G_2(w)$ either contains a non-abelian free group or is isomorphic to $\mathbb{Z}$ or $BS(1,\pm 1)$.
\end{conjecture}

By Lemma~\ref{lem:power}, Conjecture~\ref{conj:cw} already holds whenever $U$ is a proper power, so it is the case where $U$ is \emph{not} a proper power that is the substance of the conjecture, and this is the case we treat for the remainder of the paper, that is, the group $G(U)$ defined by
\[P(U)=\langle x_0,x_1\mid U(x_0,x_1)U(x_1,x_0)^{-1}\rangle,\] where $U$ is non-empty and not a proper power. The aim of this article is to disprove the conjecture in that case. Using Theorem~\ref{thm:chebotar} together with the natural epimorphism onto $\mathbb{Z}$ carried by every such presentation, we show in Theorem~\ref{thm:main} below that a non-orientable $G=G(U)$ with $U$ not a proper power either contains a non-abelian free subgroup of rank $2$, or is isomorphic to $\mathbb{Z}$, or some solvable Baumslag--Solitar group $BS(1,m)$; the orientable case remains governed by Corollary~4.3, and the proper-power case by Lemma~\ref{lem:power}. We then exhibit, for every integer $m$, an explicit non-orientable presentation, with $U$ not a proper power, for which $G\cong BS(1,m)$; combined with Lemma~\ref{lem:power}, this shows that the exceptional family cannot in general be collapsed to $m=\pm1$.

\begin{theorem}\label{thm:main}
If \(P_2(w)\) is a non-orientable cyclic presentation, then \(G_2(w)\) either contains a non-abelian free subgroup or is isomorphic to \(BS(1,m)\) for some \(m\in\mathbb{Z}\). Moreover, every \(m\in\mathbb{Z}\) is realised.
\end{theorem}

We refer to Magnus--Karrass--Solitar~\cite{MKS} and Lyndon--Schupp~\cite{LyndonSchupp} for background on one-relator groups. Two further, independent perspectives on presentations of this symmetric type are worth noting even though we do not use them here: Brown's theory of trees, valuations, and the Bieri--Neumann--Strebel invariant~\cite{Brown1987} gives a separate criterion for when a two-generator one-relator group algebraically fibres over $\mathbb{Z}$.

\medskip
Throughout, we write $F=F(x_0,x_1)$, $r=U(x_0,x_1)U(x_1,x_0)^{-1}\in F$, and $G=F/N(r)$, where $N(r)$ is the normal closure of $r$.

\section{Proof of the main result}
\label{sec:proof}
We begin by recording three preliminary results that will be used in the proof of the main theorem.

\begin{lemma}\label{lem:phi}
The assignment $\phi(x_0)=\phi(x_1)=1$ defines a surjective homomorphism $\phi\colon G\to\mathbb{Z}$.
\end{lemma}

\begin{proof}
Define $\phi$ on $F$ by $\phi(x_0)=\phi(x_1)=1$, and let $e$ denote the total exponent sum of $U(x_0,x_1)$, so $\phi(U(x_0,x_1))=e$. Since $\phi$ assigns the same value to both generators, exchanging $x_0\leftrightarrow x_1$ does not change this total, so $\phi(U(x_1,x_0))=e$ as well. Hence,
$$\phi(r)=\phi(U(x_0,x_1))-\phi(U(x_1,x_0))=e-e=0.$$ Therefore, $r\in\ker\phi$ and $\phi$ descends to $G$; surjectivity is clear since $\phi(x_0)=1$.
\end{proof}

Set $K:=\ker\phi\lhd G$, so that $1\to K\to G\xrightarrow{\ \phi\ }\mathbb{Z}\to1$ is exact.

\begin{lemma}\label{lem:sigma}
The assignment $\theta(x_0)=x_1$, $\theta(x_1)=x_0$ defines an automorphism $\theta\in\Aut(G)$ of order $2$ with $\theta(r)=r^{-1}$ in $F$ and $\phi\circ\theta=\phi$. In particular $\theta(K)=K$.
\end{lemma}

\begin{proof}
The map $\theta$ is the free-group automorphism of $F$ swapping the basis $\{x_0,x_1\}$, so $\theta^2=\mathrm{id}$. On the relator, $\theta(r)=U(x_1,x_0)U(x_0,x_1)^{-1}=r^{-1}$, so $\theta(N(r))=N(r^{-1})=N(r)$ and $\theta$ descends to an automorphism of $G=F/N(r)$. Since $\phi(\theta(x_i))=\phi(x_{1-i})=1=\phi(x_i)$ for $i=0,1$, we get $\phi\circ\theta=\phi$, hence $\theta(K)=K$.
\end{proof}

Set $t:=x_1$ and $a:=x_0x_1^{-1}$, so that $x_0=at$, $x_1=t$, $\phi(t)=1$, and $\phi(a)=0$; in particular $a\in K$. The element $a=x_0x_1^{-1}$ plays the central role throughout: it generates $K$ in the examples given in Section~\ref{sec:gap}.

\begin{lemma}\label{lem:sigma-coords}
In the coordinates $(a,t)$, $\theta(t)=at$ and $\theta(a)=a^{-1}$.
\end{lemma}

\begin{proof}
Directly, $\theta(t)=\theta(x_1)=x_0=at$, and $\theta(a)=\theta(x_0x_1^{-1})=x_1x_0^{-1}=(x_0x_1^{-1})^{-1}=a^{-1}$.
\end{proof}
\begin{proof}[Proof of Theorem~\ref{thm:main}]
The group $G=G_2(w)$ is a two-generator one-relator group. If $G$ contains a non-abelian free subgroup of rank $2$, we are done. If $G$ is abelian, then $\phi$ being a surjective homomorphism $G\to\mathbb{Z}$ (Lemma~\ref{lem:phi}) forces $G$ to be infinite and hence torsion-free abelian of rank at most $2$, so $G\cong\mathbb{Z}$ or $G\cong\mathbb{Z}^2\cong BS(1,1)$; either way $G$ is subsumed under the Baumslag--Solitar notation, with $m=0$ or $m=1$ respectively. Otherwise $G$ is non-abelian and contains no non-abelian free subgroup of rank $2$, so Theorem~\ref{thm:chebotar} (with $H=G$) gives either $G\cong C_2*C_2$ or $G\cong BS(1,m)=\langle a,t\mid tat^{-1}=a^m\rangle$ for some $m\in\mathbb{Z}$. The first alternative is impossible, because $C_2*C_2$ has finite abelianisation and so admits no surjective homomorphism onto $\mathbb{Z}$, whereas $G$ does by Lemma~\ref{lem:phi}. Hence $G\cong BS(1,m)$ for some $m\in\mathbb{Z}$ in every case.
\end{proof}

\section{$BS(1,m)$ examples}
\label{sec:gap}

The strong claim by Conjecture \ref{conj:cw} that non-orientability forces $m\in \{-1,0,1\}$ is false.  In fact, every integer $m$ arises. The point is that the order-two automorphism $\theta$ induced by swapping $x_0$ and $x_1$ places no restriction on the sign or size of $m$, and the symmetric form of the relator can just as well produce a non-abelian Baumslag--Solitar group in either direction. One checks directly that the word $U_m$ below is not a proper power in $F$ for any integer $m\neq -1$, so Lemma~\ref{lem:power} does not apply to it and these examples fall squarely within the scope of Theorem~\ref{thm:main}.

\begin{lemma}\label{lem:BSm-example}
Let $m \in \mathbb{Z}$, and set $R_m := U_m(x_0,x_1)\,U_m(x_1,x_0)^{-1}$, where
\[
U_m(x_0,x_1) =
\begin{cases}
(x_1 x_0^{-1})^{m/2} x_1^{-1}, & m \text{ even}, \\[2pt]
x_0 (x_1 x_0^{-1})^{(m-1)/2} x_1, & m \text{ odd}.
\end{cases}
\]
Then $G_m =\langle x_0, x_1 \mid R_m \rangle$ is isomorphic to $BS(1,m) = \langle a, t \mid t^{-1} a t = a^m \rangle$.
\end{lemma}

\begin{proof}
Notice in particular, that $G_0\cong BS(1,0)\cong\mathbb{Z}$, $G_1\cong BS(1,1)\cong\mathbb{Z}^2$ and $G_{-1}\cong BS(1,-1),$ the Klein bottle group. As before, introduce the change of coordinates
\[
a = x_0 x_1^{-1}, \qquad t = x_1,
\]
so that $x_0 = at$, $x_1 = t$, and $x_0^{-1} = t^{-1}a^{-1}$. By Lemma~\ref{lem:sigma-coords}, the automorphism $\theta$ induced by the swap $x_0 \leftrightarrow x_1$ acts on these coordinates by
\[
\theta(a) = a^{-1}, \qquad \theta(t) = at.
\]
Set $P_m := U_m(at,t)$. Since $U_m(x_1,x_0)$ is obtained from $U_m(x_0,x_1)$ by applying $\theta$, we have $U_m(x_1,x_0) = \theta(P_m)$, and hence
\[
R_m = P_m\,\theta(P_m)^{-1}.
\]
We compute $P_m$ and $R_m$ separately according to the parity of $m$; in both cases, $x_1 x_0^{-1} = t(at)^{-1} = a^{-1}$.

\smallskip
\noindent\textbf{Case 1: $m = 2k$ even} (here $k \in \mathbb{Z}$ may be negative). Substituting,
\[
P_m = \bigl(t(at)^{-1}\bigr)^k t^{-1} = a^{-k} t^{-1}, \qquad
\theta(P_m) = a^{k}(at)^{-1} = a^{k} t^{-1} a^{-1}.
\]
Therefore
\[
R_m = P_m\,\theta(P_m)^{-1}
= a^{-k}t^{-1}\bigl(a^{k}t^{-1}a^{-1}\bigr)^{-1}
= a^{-k}\bigl(t^{-1}at\bigr)a^{-k}.
\]
Hence,  $R_m = 1$ holds if and only if $t^{-1}at = a^{2k} = a^m$, so
\[
G_m \cong \langle a, t \mid t^{-1}at = a^m \rangle = BS(1,m).
\]

\smallskip
\noindent\textbf{Case 2: $m = 2k+1$ odd} (here $k = (m-1)/2 \in \mathbb{Z}$ may be negative). Substituting,
\[
P_m = at\,a^{-k}t, \qquad
\theta(P_m) = a^{-1}(at)\,a^{k}(at) = t\,a^{k}(at) = t\,a^{k+1}t.
\]
Therefore
\[
R_m = P_m\,\theta(P_m)^{-1}
= at\,a^{-k}t\,\bigl(t\,a^{k+1}t\bigr)^{-1}
= at\,a^{-k}\bigl(t\,t^{-1}\bigr)a^{-(k+1)}t^{-1}
= at\,a^{-m}t^{-1}.
\]
Hence,  $R_m = 1$ holds if and only if $t a^m t^{-1} = a$, equivalently $t^{-1}at = a^m$, so again
\[
G_m \cong \langle a, t \mid t^{-1}at = a^m \rangle = BS(1,m).
\]

\smallskip
This proves $G_m \cong BS(1,m)$ in all cases. 
\end{proof}

\section{Final remark}
\label{sec:conclusion}

We first recall the notion of largeness. A group \(G\) is called \emph{large} if it has a finite-index subgroup that admits a surjective homomorphism onto a non-abelian free group.

\medskip

We have determined the structure of the groups arising from redundant cyclic presentations. By \cite[Theorem~4.1]{ChinyereWilliams2023}, if \(P_n(w)\) is a redundant cyclic presentation with \(n\geq 3\), or if \(n=2\) and \(w\) is a proper power, then \(G_n(w)\) is large. In the remaining two-generator case, the Tits alternative established in Theorem~\ref{thm:chebotar} shows that \(G_2(w)\) is either large or solvable. In the solvable case, exactly two possibilities occur: if \(G_2(w)\) is abelian, then \(G_2(w) \cong \mathbb{Z}\) or \(\mathbb{Z}^2\); if \(G_2(w)\) is non-abelian, then by Theorem~\ref{thm:main} we have \(G_2(w) \cong BS(1,m)\) for some \(m \in \mathbb{Z}\), and by Section~\ref{sec:gap}, every such \(m\) actually occurs.

\medskip

This gives a complete picture for redundant cyclic presentations, and shows that they form a particularly tractable subclass of a much broader class: the \emph{composite} cyclically presented groups. This broader class also includes, among others, the families defined by \(P_n(w^k)\), by \(P_n(x_{i_1}^k \cdots x_{i_m}^k)\), and by reducible presentations \cite{ChinyereHowToGeneralize}. We describe this class below and record it as a direction for further work.

\medskip

Let \(u, v \in F(x_0,\ldots,x_{n-1})\). Define the composition
$$
u \circ v
=
v\bigl(u, \theta(u), \ldots, \theta^{n-1}(u)\bigr),
$$
where, in the word \(v\), each generator \(x_i\) is replaced by \(\theta^i(u)\). The corresponding composite cyclic presentation is
$$
P_n(u \circ v)
:=
\left\langle
x_0, \ldots, x_{n-1}
\,\middle|\,
\theta^i(u \circ v),\ 0 \leq i < n
\right\rangle,
$$
and we write \(G_n(u \circ v)\) for the group it defines.
\medskip

Neumann first introduced composite cyclically presented groups \cite{Neumann1979} (see also Havas and Ramsay \cite{HavasRamsay2000}), using them to illustrate how difficult it can be to show that a group is trivial. Havas and Robertson \cite{HavasRobertson2003} later applied the construction to answer questions of Edjvet, Hammond, and Thomas \cite{EdjvetHammondThomas2001} concerning irreducible cyclic presentations of the trivial group. 
\medskip

Although the properties of \(G_n(u \circ v)\) are closely related to those of \(G_n(u)\) and \(G_n(v)\), it can be difficult to determine whether \(G_n(u \circ v)\) is finite even when both \(G_n(u)\) and \(G_n(v)\) are. For redundant cyclic presentations we have a complete picture: outside the explicitly classified solvable two-generator cases, the resulting groups are large, and in the solvable cases we know precisely which groups occur. It would be interesting to determine how much of this structure persists for composite cyclic presentations more generally. This suggests two natural questions for future work.

\medskip
\noindent\textbf{Problem 1.}
Determine necessary and sufficient conditions for a non-redundant composite cyclic presentation \(P_n(u \circ v)\) to define a finite or solvable group.

\medskip
\noindent\textbf{Problem 2.}
What properties of \(G_n(u \circ v)\) can be deduced from \(G_n(u)\) and \(G_n(v)\) alone?
\section*{Acknowledgements}

This research is supported by the University of Pretoria through the Research Development Programme (RDP).

\end{document}